\documentclass[12pt]{amsart}

\usepackage{geometry} 
\usepackage{amssymb}
\usepackage{amscd}
\usepackage{amsmath}
\usepackage[all,cmtip]{xy}
\usepackage{mathrsfs}
\usepackage{amsthm}
\newfont{\Bb}{msbm10 scaled1200}
\newfont{\bb}{msbm8 scaled1200}
\newtheorem{thm}{Theorem}[section]
 \newtheorem{cor}[thm]{Corollary}
 \newtheorem{lem}[thm]{Lemma}
 \newtheorem{prop}[thm]{Proposition}
 \theoremstyle{definition}
 \newtheorem{defn}[thm]{Definition}
 \theoremstyle{remark}
 \newtheorem{rem}[thm]{Remark}
 \newtheorem*{ex}{Example}
 \numberwithin{equation}{section}
\begin{document}
\title[Fixed Point Theorems in multiplicative metric spaces]{Fixed points of multiplicative contraction mappings on multiplicative metric spaces}
 \author[Muttalip \"{O}zav\c{s}ar ]{Muttalip \"{O}zav\c{s}ar}
\address{Department
of Mathematics \\ 
Yildiz Technical University \\
Davutpasa-Esenler\\
P.O. Box 34210 \\
Turkey}

\email{minekci@yildiz.edu.tr}

\author[Adem C. \c{C}evikel ]{Adem C. \c{C}evikel}
\address{Department
of Mathematics Education \\
Yildiz Technical University \\
Davutpasa-Esenler\\
P.O. Box 34210 \\
Turkey}
\email{acevikel@yildiz.edu.tr}
\begin{abstract}
In this paper, we first discussed multiplicative metric mapping by giving some topological properties of the relevant multiplicative metric space. As an interesting result of our discussions, we
observed that the set of positive real numbers $\mathbb{R}_+$ is a complete multiplicative metric space with respect to the multiplicative absolute value function.
Furthermore, we introduced concept of multiplicative contraction mapping and proved some fixed point theorems of such mappings on complete multiplicative metric spaces.
\end{abstract}
\pagenumbering{arabic}
\maketitle

\section{Introduction}
Let $X$ be a nonempty set. Multiplicative metric \cite{Bashirov} is a mapping $d:X\times X\rightarrow \mathbb{R}$ satisfying the following conditions:

\noindent (m1) $d(x,y)> 1$ for all $x,y \in X$ and $d(x,y)=1$ if and only if $x=y$,

\noindent (m2) $d(x,y)=d(y,x)$ for all $x,y \in X$,

\noindent (m3) $d(x,z)\leq d(x,y)\cdot
d(y,z)$ for all $x,y,z \in X$ (multiplicative triangle inequality).

\noindent In Section 2, we investigate multiplicative metric space by remarking
its topological properties. It is well known that the set of positive real numbers $\mathbb{R}_+$ is not complete according to the usual metric. To emphasize the importance of this study, we should first note that  $\mathbb{R}_+$ is a complete multiplicative metric space with respect to the multiplicative metric. 
Furthermore,  in Section 3, we introduce concept of multiplicative contraction mapping and prove some fixed point theorems of multiplicative contraction  mappings
on multipliactive metric spaces.
 
The study of fixed points of mappings satisfying certain contraction conditions has many applications and has been at the centre of various research activities \cite{Nadler, Takahashi, Suzuki, Mohiuddine, Soon, Kimura, Agarwal, Kimuraa}.
\section{Multiplicative Metric Topology}
We shall begin this section with the following examples, which will be used throughout the paper:

\begin{ex}
Let $\mathbb{R}_{+}^{n} $ be the collection of all $n$-tubles of positive real numbers. Let $d^*:\mathbb{ R}_{+}^{n}\times \mathbb{ R}_{+}^{n}\rightarrow \mathbb{R}$ be defined
as follows
\begin{equation*}
d^{\ast }\left( x,y\right) =\left\vert \frac{x_{1}}{y_{1}}\right\vert ^{\ast
}\cdot\left\vert \frac{x_{2}}{y_{2}}\right\vert ^{\ast }\cdot \cdot \cdot \left\vert \frac{x_{n}
}{y_{n}}\right\vert ^{\ast },
\end{equation*}
where $x=\left( x_{1},...,x_{n}\right) $, $y=\left( y_{1},...,y_{n}\right) \in \mathbb{ R}_{+}^{n} $
and $\left\vert~ . ~\right\vert ^{\ast }:\mathbb{R}_{+}\rightarrow \mathbb{R}_{+}$ is defined as follows%
\begin{equation*}
\left\vert a\right\vert ^{\ast }=\left\{ 
\begin{array}{c}
a\text{ \ \ \ \ if }a\geq 1 \\ 
\frac{1}{a}\text{ \ \ \ \ if }a<1
\end{array}
\right.
\end{equation*}
Then it is obvious that all conditions of multiplicative metric are satisfied.
\end{ex}

\noindent Note that this multiplicative metric is a generalization of  the multiplicative metric on $\mathbb{R}_+$ introduced in \cite{Bashirov}.
\begin{ex}
Let $a>1$ be a fixed real number. Then $d_a:\mathbb{R}^{n}\times 
\mathbb{R}^{n}\rightarrow 
\mathbb{R}
$ defined by 
\begin{equation*}
d_{a}(x,y):=\left\vert x-y\right\vert
_{a}:=\prod_{i=1}^{n}\left\vert \frac{a^{x_{i}}}{a^{y_{i}}}
\right\vert ^{\ast }
\end{equation*} holds the multiplicative metric conditions.
\end{ex} 
\noindent Note that $d_a$ satisfies the following relation:
\begin{equation*}
\prod\limits_{i=1}^{n}\left\vert \frac{a^{x_{i}}}{a^{y_{i}}}\right\vert
^{\ast }=a^{\sum\limits_{i=1}^{n}\left\vert x_{i}-y_{i}\right\vert}.
\end{equation*}
One can extend this multiplicative metric to $
\mathbb{C}^{n}$ by the following definition:
\begin{equation*}
d_{a}\left( \omega ,z\right) =\left\vert \omega -z\right\vert
_{a}=a^{\sum\limits_{i=1}^{n}\left\vert \omega _{i}-z_{i}\right\vert }.
\end{equation*}

The following proposition plays a key role in what follows.
\begin{prop}
(Multiplicative reverse triangle inequality). Let $( X, d) $ be a
multiplicative metric space. Then we have the following inequality
\begin{equation*}
\left\vert \frac{d(x,z)}{d(y,z)}\right\vert ^{\ast }\leq d(x,y)
\end{equation*}
for all $x,y,z\in X$.
\end{prop}
\begin{proof}
By the multiplicative triangle inequality, it is seen that 
\begin{equation}
\begin{aligned}
&d(x,z)\leq
d(x,y)\cdot d(y,z)\Rightarrow \frac{d(x,z)}{d(y,z)}\leq d(x,y),\\
&d(y,z)\leq d(y,x)\cdot d(x,z)\Rightarrow \frac{1}{d(x,y)}\leq \frac{d(x,z)}{d(y,z)
}.
\end{aligned}
\end{equation}
 Thus, by the definition of $\left\vert~.~ \right\vert ^{\ast }$ and (2.1), we have
\begin{eqnarray}
\frac{1}{d(x,y)}\leq \frac{d(x,z)}{d(y,z)}\leq d(x,y)\Leftrightarrow
\left\vert \frac{d(x,z)}{d(y,z)}\right\vert ^{\ast }\leq d(x,y)\nonumber.
\end{eqnarray}
\end{proof}
\begin{defn}
(Multiplicative balls): Let $(X,d)$ be a multiplicative
metric space, $x\in X$ and $\varepsilon > 1$. We now define a set 
\begin{equation*}
B_{\varepsilon }(x)=\left\{y\in X\mid d(x,y)<\varepsilon \right\},
\end{equation*}
which is called multiplicative open ball of radius $\varepsilon $ with center $x$. Similarly, one can describe multiplicative closed ball as 
\begin{equation*}
\overline{B}_{\varepsilon }(x)=\left\{y\in X\mid d(x,y)\leq \varepsilon \right\} \nonumber.
\end{equation*}
\end{defn}
 \begin{ex}
Consider the multiplicative metric space $(\mathbb{R_+},d^*)$.
From the definition of $d^*$, we can verify that multiplicative open ball of radius $\varepsilon>1$ with center
$x_0$ appears as $(\frac{x_0}{\varepsilon}, x_0\cdot \varepsilon)\subset \mathbb{R_+}$.
\end{ex}
\begin{defn}
(Multiplicative interior point): Let $(X, d)$ be a multiplicative metric space and  $A\subset X$. Then we call $x\in A$ a multiplicative
interior point of $A$ if there exists an $\varepsilon >1$ such that $%
B_{\varepsilon }(x)\subset A$. The collection of all interior points of $A$
is called multiplicative interior of $A$ and denoted by $int(A)$.
\end{defn}
\begin{defn}
(Multiplicative open set):
Let $(X, d)$ be a multiplicative metric space and  $A\subset X$. If every point of $A$ is a multiplicative interior point of $A$, i.e., $
A=int(A)$, then $A$ is called a multiplicative open set. 
\end{defn}
\begin{lem}
Let $(X, d)$ be a multiplicative metric space. Each multiplicative open ball  of $X$ is a multiplicative open set.
\end{lem}
\begin{proof}
Let $x\in X$ and $B_{\varepsilon}\left( x\right)$ be a multiplicative open ball. For $y \in B_{\varepsilon}\left( x\right)$, if we let $\delta= \frac{
\varepsilon }{d(x,y)}$  and $z\in B_{\delta}(y )$, then $d(y,z)<\frac{\varepsilon }{d(x,y)}$, from
which we conclude that 
\begin{equation*}
d(x,z)<d(x,y)\cdot d(y,z)<\varepsilon.
\end{equation*}
This shows that $z\in B_{\varepsilon}\left( x\right)$, which means that $B_{\delta}(y )\subset B_{\varepsilon}\left( x\right)$. Thus $B_{\varepsilon}\left( x\right)$ is multiplicative open set.
\end{proof}
\begin{lem}
 Let $(X, d)$ be a multiplicative metric space. Then $X$ and $\emptyset$ are multiplicative open sets.
\end{lem}
\begin{lem}
The union of any finite, countable or uncountable family of  multiplicative open sets is also  a multiplicative open set.
\end{lem}
\begin{lem}
The intersection of any finite family of multiplicative open sets is also a multiplicative open set.
\end{lem}
\begin{proof}
 Let $B_{1}$ and $B_{2}$
be two multiplicative open sets and $y\in B_{1}\cap B_{2}$. Then there are  $\delta _{1}, \delta _{2}>1$ such that $%
B_{\delta _{1}}(y)$ $\subset $ $B_{1}$ and $B_{\delta _{2}}(y)$ $\subset $ $%
B_{2}.$ Letting $\delta $ be the smaller of $\delta _{1}$ and $\delta _{2}$,
we conclude that $B_{\delta}(y )\subset B_{1}\cap B_{2}$. Hence the intersection of any finite family of multiplicative open sets is a multiplicative open set.
\end{proof}
\begin{thm}
Every multiplicative metric space is a topological space based on the set of all multiplicative open sets.
\end{thm}
\begin{proof}
By the results obtained until now, it is clear. 
\end{proof}
\begin{defn}
Let $(X,d)$ be a multiplicative metric space.  A point $x\in X$ is said to be multiplicative limit point of $S\subset X$ if and only if $\left( B_{\varepsilon}(x)\setminus \left\{x\right\} \right)\cap S\neq \emptyset$ for every $\varepsilon>1$. The set of all multiplicative limit points of the set $S$ is denoted by $S'$.
\end{defn}
\begin{defn}
Let $(X,d) $ be a multiplicative metric space. We call a set $S\subset X$ multiplicative closed in $(X,d)$ if $S$ contains all of its multiplicative limit points.
\end{defn}

The following propositions can be easily proven by the definition of multiplicative closed set:
\begin{prop}
Let $(X,d) $ be a multiplicative metric space and $S\subset X$. Then $S\cup S'$ is a multiplicative closed set. This set is called multiplicative closure of the set $S$, which is denoted by $\overline{S}$.
\end{prop}
\begin{prop}
Let $(X, d)$ multiplicative metric space and $S\subset X$. $S$ is multiplicative closed if and only if $X\setminus S$, the complement of $S$, is multiplicative open.
\end{prop}
\begin{defn}
(Multiplicative continuity): Let $( X,d_X) $ and $( Y,d_Y) $ be two multiplicative metric spaces and $f:X \rightarrow Y$ be a function. If $f$
holds the requirement that, for every $\varepsilon >1 $, there exists $ \delta >1$ such that $f(B_{\delta}(x))\subset B_{\varepsilon}(f(x))$, then we call $f$ multiplicative continuous at $x\in X$.
\end{defn}
\begin{ex}
Given a multiplicative metric space $(X, d)$, define a multiplicative metric on $X\times X$ by
\begin{equation*}
\rho(( x_1,x_2),(y_1,y_2))=d(x_1,y_1)\cdot d(x_2,y_2).
\end{equation*}
Then the multiplicative metric $d:X\times X\rightarrow (\left[1,\infty\right), | . |^*)$ is multiplicative continuous on $X\times X$. To show this, let $(y_1, y_2), (x_1,x_2) \in X \times X $. Since we have
\begin{equation*}
\left\vert \frac{d(y_1,y_2)}{d(x_1,x_2)}\right\vert ^{\ast }\leq d(x_1,y_1)\cdot d(x_2,y_2),
\end{equation*}
it is clear that $d$ is multiplicative continuous on $X\times X$.
\end{ex}
\begin{defn}
(Semi-multiplicative continuity): Let $( X,d)$ be
a multiplicative metric space, $( Y,d)$ be a metric
space and $f:X \rightarrow Y$ be a
function. If $f$ holds the requirement that, for all $\varepsilon >0$,
there exists $ \delta >1$ such that $f(B_{\delta}(x))\subset B_{\varepsilon}(f(x))$, then we call $f$ semi-multiplicative continuous at $x\in X$. Similarly, a function $g:Y
\rightarrow X $ is also said to be
semi-multiplicative continuous at $y\in Y$if it satisfies a similar requirement. 
\end{defn}
\begin{ex}
Let $f:\left(\mathbb{R}_+, \left|.\right|^*\right)\rightarrow\left(\mathbb{R}, \left|.\right|\right) $ be a function defined as $f(x)=ln(x)$. \newline  Let ${\varepsilon>0}$ and 
$\left|ln(x)-ln(y)\right| <\varepsilon $. If we let $\delta=e^{\varepsilon}$, then we have $\left| \frac{x}{y}\right|^*<\delta$, which implies that $f$ is semi-multiplicative continuous on $\mathbb{R}_+$.
\end{ex}
\begin{ex}
Let $[a, b]\subset \mathbb{R}$ and $C^{\ast }[a,b]$ be the collection of all semi-multiplicative continuous
functions from $[a,b]$ to $\mathbb{R}_{+}$. Then the following mapping is a multiplicative metric on $C^{\ast }[a,b]$:
\begin{eqnarray}
d(f,g)=\max_{x\in \lbrack a,b]}\left\vert \frac{f(x)}{g(x)}
\right\vert ^{\ast } ;~ f,g \in C^{\ast }[a,b] .\nonumber 
\end{eqnarray}
\end{ex}

\begin{defn} (Multiplicative convergence):
Let $(X, d)$ be a multiplicative metric space, $(x_n)$ be a sequence in $X$ and $x\in X$.~If for every multiplicative open ball $B_{\varepsilon}(x )$,
there exists a natural number $N$ such that $n\geq N\Rightarrow x_{n}\in B_{\varepsilon}(x )$, then the sequence $(x_{n})$
is said to be multiplicative convergent to $x$, denoted by $
x_n\rightarrow_{*} x~(n\rightarrow \infty)$. 
 \end{defn}
\begin{lem}
Let $(X, d)$ be a multiplicative metric space, $(x_n)$ be a sequence in $X$ and $x\in X$. Then \newline
 $x_n\rightarrow_* x~(n\rightarrow \infty) $ if and only if $d(x_n,x)\rightarrow _*1~(n\rightarrow \infty )$. 
\end{lem}
\begin{proof}
Suppose that the sequence $(x_n)$ is multiplicative convergent to $x$. That is, for every $\varepsilon > 1$, there is a natural number $N$ such that $d(x_n,x)< \varepsilon$ whenever $n\geq N$. Thus
we have the following inequality
\begin{equation*}
\frac{1}{\varepsilon}< d(x_n,x)< 1\cdot \varepsilon~\text{for all}~ n \geq N.
\end{equation*}
 This means $|d(x_n,x)|^{*}<\varepsilon$ for all $n\geq N$, which implies that the sequence $d(x_n,x)$ is multiplicative convergent to 1.

It is clear to verify the converse.
\end{proof}
\begin{lem}
Let $(X, d)$ be a multiplicative metric space, $(x_n)$ be a sequence in $X$. If the sequence $(x_n)$ is multiplicative convergent, then the multiplicative limit point is unique.
\end{lem}
\begin{proof}
Let $x,y\in X$ such that $x_n\rightarrow_* x$ and $x_n\rightarrow_* y~(n\rightarrow \infty)$. That is, for every $\varepsilon >1$, there exists $N$ such that, for all $n\geq N$, we have $d(x_n,x)< \sqrt{\varepsilon}$ and  $d(x_n,y)< \sqrt{\varepsilon}$. Then we have
\begin{equation*}
d(x,y)\leq d(x_n,x)\cdot d(x_n,y)<\varepsilon
\end{equation*}
Since $\varepsilon$ is arbitrary, $d(x,y)=1$. This says $x=y$.
\end{proof}
\begin{thm}
Let $(X, d_X)$ and $(Y, d_Y)$ be two multiplicative metric spaces, $f:X\rightarrow Y$ be a mapping and $(x_n)$ be any sequence in $X$. Then f is multiplicative continuous at the point $x\in X$ if and only if  $f(x_n)\rightarrow_{*} f(x)$ for every sequence $(x_n)$ with $x_n\rightarrow_{*}x~(n\rightarrow \infty)$.
\end{thm}

\begin{proof} 
Suppose that $f$ is multiplicative continuous at the point $x$ and $x_n\rightarrow_{*}x$. From the multiplicative continuity of $f$, we have that, for every $\varepsilon>1$, there exists $\delta>1$ such that $f\left(B_{\delta}(x)\right)\subset B_{\varepsilon}(f(x))$. Since $x_n\rightarrow_{*}x~(n\rightarrow \infty)$, there exists $N$ such that $n\geq N$ implies $x_n\in B_{\delta}(x)$. By virtue of the above inclusion, then  $f(x_n)\in B_{\varepsilon}(f(x))$ and hence $f(x_n)\rightarrow_{*} f(x)~(n\rightarrow \infty)$.

Conversely, assume that $f$ is not multiplicative continuous at $x$. That is, there exists an $\varepsilon>1$ such that, for each $\delta>1$,  we have $x' \in X$ with $d_X(x', x)<\delta$ but
\begin{equation}
 d_Y(f(x'), f(x))\geq \varepsilon. 
\end{equation}
Now, take any sequence of real numbers $(\delta_n)$ such that $\delta_n\rightarrow 1$ and $\delta_n>1$ for each $n$. For each $n$, select $x'$ that satisfies the equation (3) and 
call this $x_n'$.~It is clear that $x_n'\rightarrow_{*}x$, but $f(x_n')$ is not multiplicative convergent to $ f(x)$. Hence we see that if $f$ is not  multiplicative continuous, then not every sequence $(x_n)$ with $x_n\rightarrow_* x$ will yield a sequence $f(x_n)\rightarrow_{*}f(x)$. Taking the contrapositive of this statement demonstrates that the condition is sufficient.
\end{proof}

Similarly, we can prove the following theorem.
\begin{thm}
Let $(X, d_X)$ and $(Y, d_Y)$ be the usual metric space and multiplicative metric space, respectively. Let $f:X\rightarrow Y$ be a mapping and $(x_n)$ be any sequence in $X$. Then f is semi-multiplicative continuous at the point $x\in X$ if and only if $f(x_n)\rightarrow_{*} f(x)$ for every sequence $(x_n)$ with $x_n \rightarrow x~(n\rightarrow\infty)$.
\end{thm}



\begin{thm}
Let $(X,d)$ be a multiplicative metric space and $S\subset X$. Then \newline
(i) A point $x\in X$ belongs to $\overline{S}$ if and only if there exists a sequence $(x_n)$ in $S$ such that $x_n\rightarrow_* x~(n\rightarrow \infty)$.\newline
(ii) The set $S$ is multiplicative closed if and only if every multiplicative convergent sequence in $S$ has a multiplicative limit point that belongs to $S$.
\end{thm}
\begin{proof} 

$(i)$ Let $x\in S$ . If we consider a sequence $(x_n)$ with $x_n=x$ for all $n$,  then this is a sequence in $S$ such that $x_n\rightarrow_* x$. Let $x\in S'$. Thus, for 
$\varepsilon_n=1+\frac{1}{n}$, we have  $B_{\varepsilon_n}(x)\cap S\neq \emptyset$. By choosing $x_n\in B_{\varepsilon_n}(x)\cap S$, we can set a sequence $(x_n)$ in $S$ such that
$x_n\rightarrow_* x~(n\rightarrow \infty)$.

It is easy to prove the converse. \newline
$(ii)$ From $(i)$, it is easily seen.
\end{proof}
\begin{defn}
Let $(X,d)$ be a multiplicative metric space and $(x_{n})$ be a sequence in $X$. The sequence is called a multiplicative Cauchy sequence if it holds that, for all $\varepsilon >1$, there exists $N\in \mathbb{N}$ such that $d(x_{m},x_{n})<\varepsilon$ for $m, n\geq N$.
\end{defn}
\begin{thm}
Let $(X,d)$ be a multiplicative metric space and $(x_{n})$ be a sequence in $X$. The sequence is multiplicative convergent, then it is a multiplicative Cauchy sequence.
\end{thm}
\begin{proof}
Let $x\in X$ such that $x_n\rightarrow_* x$. Hence we have that 
for any $\varepsilon >1,$ there exist
a natural number $N$ such that $d(x_{n},x)<\sqrt{\varepsilon }$ and $d(x_{m},x)<\sqrt{\varepsilon }$ for all $
m,n\geq N$. By the multiplicative triangle inequality, we get
\begin{equation*}
d(x_{n},x_{m})\leq 
d(x_{n},x)\cdot d(x,x_{m})<\sqrt{\varepsilon }\cdot\sqrt{\varepsilon }=\varepsilon,
\end{equation*}
which implies $(x_n)$ is a multiplicative Cauchy sequence.
\end{proof}

\begin{thm}
(Multiplicative characterization of supremum)
Let $A$ be a nonempty subset of $ \mathbb{R}_+$. \newline 
Then $s=\sup{A}$ if and only if \newline
(i) $a\leq s$ for all $a\in A$ \newline
(ii) there exists at least a point $a\in A$ such that $|\frac{s}{a}|^*< \varepsilon$ for all $\varepsilon > 1$.
\end{thm}
\begin{proof}
Let $s=\sup{A}$. Then from the definition of supremum, the condition $(i)$ is trivial. To prove the condition $(ii)$, assume that there is an $\varepsilon >1$ such that there are no elements
$a \in A$ such that $| \frac{s} {a} |^*< \varepsilon$. If this is the case, then $\frac{s}{\varepsilon}$ is also an upper bound for the set $A$. But this is impossible, since $s$ is the smallest upper bound for
$A$.

To prove the converse, assume that the number $s$ satisfies the conditons $(i)$ and $(ii)$. By the condition $(i)$, $s$ is an upper bound for the set $A$. If $s\neq \sup{A}$, then 
$s> \sup{A}$ and $\varepsilon = \frac{s}{\sup{A}}> 1$. By the condition $(ii)$, there exists at least a number $a \in A $ such that $|\frac{s}{a}|^*< \varepsilon$. By the definition of the number $\varepsilon$, 
this means that $a > \sup{A}$. This is impossible, hence $s=\sup{A}$. 
\end{proof}
Similarly, we can prove the following theorem.
\begin{thm}
(Multiplicative characterization of infimum)
Let $A$ be a nonempty subset of $ \mathbb{R}_+$. \newline 
Then $m=\inf{A}$ if and only if \newline
(i) $m\leq a$ for all $a\in A$ \newline
(ii) there exists at least a point $a\in A$ such that $|\frac{a}{m}|^*< \varepsilon$ for all $\varepsilon > 1$.
\end{thm}
\begin{defn}
Let $(X, d)$ be a multiplicative metric space and $ A \subset X$. The set $A$ is called multiplicative bounded if  there exist $x\in X$ and $M> 1$
such that $A\subseteq B_M(x)$.
\end{defn}
\begin{thm}
A multiplicative Cauchy sequence is multiplicative bounded.
\end{thm}
\begin{proof}
Let $(X,d)$ be a multiplicative metric space and $(x_n)$ be a multiplicative Cauchy sequence in it. The definition of multiplicative Cauchy sequence 
implies that for $\varepsilon =2>1$, there exists a natural number $n_0$ such that $d(x_{n},x_{m})<2 $ for all $m,n\geq n_0$. Hence if 
we set 
\begin{equation*}
M=\max \{2, d(x_{1},x_{n_{0}}),...,d(x_{n_0-1},x_{n_{0}})\},
\end{equation*}
then it is clear that $d(x_{n},x_{n_{0}})<M$ for all $n\in \mathbb{N}$. Thus we have
\begin{equation*}
d(x_{n},x_{m})\leq d(x_{n},x_{n_{0}})\cdot d(x_{m},x_{n_{0}})<M^{2} ~\text{for all}~m,n\in \mathbb{N}.
\end{equation*}
 This says that the sequence is multiplicative bounded. 
\end{proof}
\begin{thm}
Let $(x_{n})$ and $(y_{n})$ be multiplicative Cauchy sequences in a
multiplicative metric space $(X,d)$. The sequence $(
d(x_{n},y_{n}))$ is also a multiplicative Cauchy sequence in the
multiplicative metric space $( \mathbb{R}_{+},d^{\ast }) .$
\end{thm}
\begin{proof}
From the multiplicative reverse triangle inequality, we have
\begin{eqnarray}
d^{\ast }\left( d(x_{n,}y_{n}),d(x_{m},y_{m})\right)  &=&\left\vert \frac{%
d(x_{n,}y_{n})}{d(x_{m},y_{m})}\right\vert ^{\ast } \nonumber \\
&\leq &\left\vert \frac{d(x_{n},y_{n})}{d(x_{m},y_{n})}\right\vert ^{\ast
}.\left\vert \frac{d(x_{m},y_{n})}{d(x_{m},y_{m})}\right\vert ^{\ast } \nonumber \\
\text{ \ \ \  \ \ \ \ } &\leq &d(x_{n},x_{m})\cdot d(y_{n},y_{m}). \nonumber
\end{eqnarray}%
Since $(x_{n})$ and $(y_{n})$ are multiplicative Cauchy
sequences,  for any $\varepsilon >1$, there exists $N\in \mathbb{N}$ such that $d(x_{n},x_{m})<\sqrt{\varepsilon }$
and $d(y_{n},y_{m})<\sqrt{\varepsilon }$ for all $n,m\geq N$. This implies $d^{\ast }\left( d(x_{n,}y_{n}),d(x_{m},y_{m})\right) <\varepsilon$ 
for all $n,m \geq N$, which says $(
d(x_{n},y_{n}))$ is a multiplicative Cauchy sequence. 
\end{proof}
\begin{lem}
Let $(X,d)$ be a multiplicative metric space and $(x_n)$ be a sequence in $X$. Then $(x_n)$ is a multiplicative Cauchy sequence if and only if $d(x_n,x_m)\rightarrow_* 1~(m,n\rightarrow \infty)$.
\end{lem}
\begin{proof}
Let $(x_n)$ be a multiplicative Cauchy sequence. Then, for every $\varepsilon >1$, there exists $N$ such that $d(x_n,x_m)< \varepsilon$ for all $m,n \geq N$. Hence,
from the definition of  $\left|.\right|^*$, we have $|d(x_n,x_m) |^*< \varepsilon$ whenever $n,m\geq N$. This means $d(x_n,x_m)\rightarrow_* 1~(m,n\rightarrow \infty)$ in $(\mathbb{R}_+,\left|.\right|^*)$.

It is easy to prove the sufficiency side of Lemma.
\end{proof}

\begin{thm}
Let $(X,d)$ be a multiplicative metric space. Let $(x_n)$ and $(y_n)$ be two sequences in $X$ such that $x_n\rightarrow_* x$, $y_n\rightarrow_* y~(n\rightarrow \infty),~x, y\in X$. Then 
\begin{equation*}
d(x_n,y_n)\rightarrow_*d(x,y)~(n\rightarrow\infty).
\end{equation*}
\end{thm}
\begin{proof}
Let $x_n\rightarrow _*x$ and $y_n\rightarrow_* y$ as $n\rightarrow \infty$. That is, for all $\varepsilon>1$, there exists $N\in \mathbb{N}$ such that, for all $n\geq N$, we have $d(x_n,x)< \sqrt{\varepsilon}$
and $d(y_n,y)< \sqrt{\varepsilon}$. On the other hand, by the multiplicative triangle inequality, we have 
\begin{eqnarray}
d(x_n,y_n)\leq d(x_n,x)\cdot d(x,y)\cdot d(y_n,y), \nonumber \\
d(x,y)\leq d(x_n,x)\cdot d(x_n,y_n)\cdot d(y_n,y), \nonumber
\end{eqnarray}
which imply
\begin{eqnarray}
\frac{d(x_n,y_n)}{d(x,y)}\leq d(x_n,x)\cdot d(y_n,y) ,\nonumber \\
\frac{d(x,y)}{d(x_n,y_n)}\leq d(x_n,x)\cdot d(y_n,y) .\nonumber
\end{eqnarray}
Hence, from the definition of $\left| .\right|^*$, we have
\begin{equation*}
\left| \frac{d(x_n,y_n)}{d(x,y)}\right|^* \leq d(x_n,x)\cdot d(y_n,y)< \varepsilon~\text{for all}~n\geq N.
\end{equation*} 
This says $d(x_n,y_n)\rightarrow_*d(x,y)~(n\rightarrow \infty)$ in $(\mathbb{R}_+,\left|.\right|^*)$.
\end{proof}

\begin{thm}
Let $(x_n)$ be a multiplicative Cauchy sequence in a multiplicative metric
space $(X,d)$. If the sequence $(x_n)$ has a subsequence $(x_{n_{k}})$ such that  $x_{n_{k}}\rightarrow_{\ast}x\in X~(n_k\rightarrow \infty)$,
then $(x_{n})\rightarrow_{\ast}x~(n\rightarrow \infty)$.
\end{thm}
\begin{proof}
Let $(x_n)$ be multiplicative Cauchy sequence and $\varepsilon>1$. Thus there exists $N\in \mathbb{N}$ such that
 $d(x_{n},x_{m})<\sqrt{\varepsilon }$ for $m,n\geq N$.
Let $(x_{n_{k}})$ be a subsequence of $(x_n)$ such that $x_{n_{k}}\rightarrow_{\ast} x$. Then there exists $k_{1} \in \mathbb{N}$ such that $d(x_{n_{k}},x)<\sqrt{\varepsilon }$
for all $k\geq k_1$. Morever, $\lim_{k\rightarrow \infty }{n_{k}}=\infty $ implies that there
exists $k_{2}\in \mathbb{N}$ such that $n_{k}\geq N$ for all $k\geq k_{2}.$ Hence we have
\begin{equation*}
d(x_{k},x)\leq d(x_{k},x_{n_{k}})\cdot d(x_{n_{k}},x)<\sqrt{\varepsilon }\cdot\sqrt{
\varepsilon }=\varepsilon \nonumber~\text{ for all }~k\geq \max \{n,k_{1},k_{2}\}.
\end{equation*}
 This shows
$x_{n}\rightarrow_{\ast} x~(n\rightarrow \infty)$.
\end{proof}
\begin{lem}
Every sequence in $\mathbb{R}$ has a monotone subsequence.
\end{lem}
\begin{lem}
A monotone  multiplicative bounded sequence in $(
\mathbb{R}_{+},\left\vert \cdot \right\vert ^{\ast })$  is multiplicative convergent.
\end{lem}
\begin{proof}
WLOG, let $(x_n)$ be increasing multiplicative bounded sequence in $\mathbb{R}_{+}$ and $x=\sup x_{n}$. Our claim is that $x_{n}\rightarrow_{*} x~(n\rightarrow \infty)
$. By the multiplicative characterization of supremum, for any $\varepsilon >1$, there exists $x_{n}$ such that $\left\vert \frac{
x_{n}}{x}\right\vert ^{\ast }<\varepsilon $. Set $n=N$. Then, by the
monotonicity, $x_{m}>x_{N}$ for all $ m>n$. But as $x_{m}<x$, this
implies that $\left\vert \frac{x_{m}}{x}\right\vert ^{\ast }<\varepsilon $ for all $ m>N$. Thus $x_{n}\rightarrow_{*} x~(n\rightarrow \infty)$.
\end{proof}

An immediate corollary of  the results obtained until now is multiplicative counterpart of
the Bolzano-Weierstrass theorem in the classical analysis.

\begin{cor}
(Multiplicative Bolzano-Weierstrass) A multiplicative bounded sequence in $(
\mathbb{R}_{+},\left\vert \cdot \right\vert ^{\ast })$
has a multiplicative convergent subsequence.
\end{cor}
\begin{defn}
We call a multiplicative metric space complete if every multiplicative
Cauchy sequence in it is multiplicative convergent to $x\in X$.
\end{defn}
\begin{ex}
All the results obtained until now indicate that $(R_{+},\left\vert
\cdot \right\vert ^{\ast })$ is complete. 
\end{ex}

\begin{thm}
Let $(X,d)$ be a complete multiplicative metric space and $S\subset X$. Then $(S,d)$ is complete if and only if $S$ is multiplicative closed.
\end{thm}

\begin{proof}
By Theorem 2.21 and Theorem 2.23, it is clear.
\end{proof}

\section{Fixed point theorems}
\begin{defn}
Let $(X, d)$ be a multiplicative metric space. A mapping $f:X\rightarrow X$ is called multiplicative contraction if there exists a real constant $\lambda \in [0, 1)$
such that
\begin{equation*}
d(f(x_1),f(x_2))\leq d(x_1,x_2)^{\lambda} \text{ for all}~ x,y\in X.
\end{equation*}
\end{defn}

Now, based on the definition of multiplicative contraction, we introduce the following Banach contraction principle for multiplicative metric spaces.
\begin{thm}
Let $(X, d)$ be a multiplicative metric space and let $f:X\rightarrow X$ be a multiplicative contraction. If $(X, d)$ is complete, then $f$ has a unique fixed point.
\end{thm}
\begin{proof}
Consider a point $x_0 \in X$. Now we define a sequence $(x_n)$ in $X$ such that $x_n=fx_{n-1}$ for $n=1,2,...$
From the multiplicative contraction property of $f$, we have
\begin{equation*}
d(x_{n+1}, x_n)\leq d(x_n,x_{n-1})^{\lambda} \leq d(x_{n-1},x_{n-2})^{ \lambda^2}\leq \cdot \cdot \cdot \leq d(x_1,x_0)^{\lambda^n}.
\end{equation*} 
Let $m, n \in \mathbb{N}$ such that $m> n$, then we get
\begin{eqnarray}
d(x_m, x_n)&\leq& d(x_m, x_{m-1})\cdot\cdot\cdot d(x_{n+1},x_n)\nonumber \\
&\leq&d(x_1,x_0)^{\lambda^{m-1}+\cdot \cdot \cdot+ \lambda^n} \nonumber \\
&\leq&d(x_1,x_0)^{\frac{\lambda^n}{1-\lambda}}. \nonumber
\end{eqnarray}
This implies that $d(x_m,x_n)\rightarrow_* 1~(m,n\rightarrow \infty)$. Hence the sequence $(x_n)=(f^nx_0)$ is multiplicative Cauchy. By the completeness of 
$X$, there is $z\in X$ such that $x_n\rightarrow_* z~(n\rightarrow \infty)$. Moreover,
\begin{equation*}
d(fz,z)\leq d(fx_n, fz)\cdot d(fx_n,z)\leq d(x_n,z)^{\lambda }\cdot d(x_{n+1},z)\rightarrow_* 1~(n\rightarrow \infty),
\end{equation*}
which implies $d(fz,z)=1$. Therefore this says that $z$ is a fixed point of $f$; that is, $fz=z$.

Now, if there is another point $y$ such that $fy=y$, then 

\begin{equation*}
d(z,y)=d(fz,fy)\leq d(z,y)^{\lambda}. 
\end{equation*}
Therefore $d(z,y)=1$ and $y=z$. This implies that $z$ is  the unique fixed point of $f$. 
\end{proof}
\begin{cor}
Let $(X, d)$ be a complete multiplicative metric space. For $\varepsilon$ with $\varepsilon >1$ and $x_0\in X$, consider the multiplicative closed ball $\overline{B}_{\varepsilon}(x_0)$. 
Suppose the mapping $f:X\rightarrow X$ satisfies the contraction condition
\begin{equation*}
d(fx,fy)\leq d(x,y)^{\lambda}~ \text{for all x,y} \in \overline{B}_{\varepsilon}(x_0)
\end{equation*}
where $\lambda \in \left[0,  1\right)$ is a constant and $d(fx_0,x_0)\leq\varepsilon^{1-\lambda}$. Then $f$ has a unique fixed point in $\overline{B}_{\varepsilon}(x_0)$.
\end{cor}
\begin{proof}
We only need to prove that $\overline{B}_{\varepsilon}(x_0)$ is complete and $fx\in\overline{B}_{\varepsilon}(x_0)$ for all $x\in\overline{B}_{\varepsilon}(x_0)$.
Suppose $(x_n)$ is a multiplicative Cauchy sequence in $\overline{B}_{\varepsilon}(x_0)$. Then $(x_n)$ is also a multiplicative Cauchy sequence in $X$. By the completeness of $X$, there exists $x\in X$ such that $x_n\rightarrow _*x$. We have 
\begin{equation*}
d(x_0,x)\leq d(x_n,x_0)\cdot d(x_n,x)\leq d(x_n,x)\cdot \varepsilon
\end{equation*}
Since $x_n\rightarrow _*x$, $d(x_n,x)\rightarrow_*1$. Hence $d(x_0,x)\leq \varepsilon$, and $x\in \overline{B}_{\varepsilon}(x_0)$. It follows that 
$\overline{B}_{\varepsilon}(x_0)$ is complete.

For every $x\in \overline{B}_{\varepsilon}(x_0)$,
\begin{equation*}
d(x_0, fx)\leq d(fx_0,x_0)\cdot d(fx_0,fx)\leq \varepsilon^{1-\lambda}\cdot d(x_0,x)^{\lambda} \leq \varepsilon^{1-\lambda}\cdot \varepsilon^{\lambda}=\varepsilon.
\end{equation*}
Thus $fx \in \overline{B}_{\varepsilon}(x_0) $.
\end{proof}

\begin{cor}
Let $(X, d)$ be a complete multiplicative metric space.~If a mapping $f:X\rightarrow X$ satisfies for some positive integer $n$,
\begin{equation*}
d(f^nx,f^ny)\leq d(x,y)^{\lambda}~\text{for all} ~x, y\in X,
\end{equation*}
where $\lambda \in \left[0, 1\right)$ is a constant, then $f$ has a unique fixed point in $X$.
\end{cor}

\begin{proof}
From Theorem 3.2, $f^n$ has a unique fixed point $z\in X$. But $f^n(fz)=f(f^nz)=fz$, so $fz$ is also a fixed point of $f^n$. Hence $fz=z$, $z$ is a fixed point of $f$. Since the fixed point of $f$
is also fixed point of $f^n$, the fixed point of $f$ is unique.
\end{proof}
\begin{thm}
Let $(X,d)$ be a complete multiplicative metric space. Suppose the mapping $f:X\rightarrow X$ satisfies the contraction condition
\begin{equation*}
d(fx, fy)\leq (d(fx, x)\cdot d(fy, y))^{\lambda}~ \text{for all}~x,y\in X,
\end{equation*}
where $\lambda \in \left[0, \frac{1}{2} \right)$ is a constant. Then $f$ has a unique fixed point in $X$. And for any $x\in X$, iterative sequence $(f^nx)$ converges to the fixed point.
 \end{thm}
\begin{proof}
Choose $x_0\in X$. Set $x_1=fx_0,x_2=fx_1=f^2x_0,...,x_{n+1}=fx_n=f^{n+1}x_0,...$

\noindent we have
\begin{eqnarray}
d(x_{n+1}, x_n)=d(fx_n,fx_{n-1})&\leq& (d(fx_n,x_n)\cdot d(fx_{n-1},x_{n-1}))^{\lambda} \nonumber \\
&=&(d(x_{n+1},x_n)\cdot d(x_n,x_{n-1}))^{\lambda}. \nonumber
\end{eqnarray}
Thus we have
\begin{equation*}
d(x_{n+1},x_n)\leq (d(x_n,x_{n-1}))^{\frac{\lambda}{1-\lambda}}= d(x_n,x_{n-1})^h, \nonumber
\end{equation*}
where $h=\frac{\lambda}{1-\lambda}$. For $n>m$,
\begin{eqnarray}
d(x_n, x_m)&\leq& d(x_n,x_{n-1})\cdot d(x_{n-1}, x_{n-2})\cdot \cdot \cdot d(x_{m+1}, x_m) \nonumber \\
&\leq&d(x_1,x_0)^{h^{n-1}+h^{n-2}+\cdot \cdot \cdot + h^m} \leq d(x_1,x_0)^{\frac{h^m}{1-h}} \nonumber 
\end{eqnarray}
This implies $d(x_n,x_m)\rightarrow_* 1(n,m\rightarrow \infty)$. Hence $(x_n)$ is a Cauchy sequence. By the completeness of $X$, 
there is $z\in X$ such that $x_n\rightarrow_* z~(n\rightarrow \infty)$. Since 
\begin{eqnarray}
d(fz, z)&\leq& d(fx_n, fz)\cdot d(fx_n, z) \nonumber \\
&\leq& (d(fx_n, x_n)\cdot d(fz, z))^{\lambda} \cdot d(x_{n+1}, z), \nonumber
\end{eqnarray}
we have
\begin{equation*}
d(fz, z)\leq (d(fx_n,x_n)^{\lambda}\cdot d(x_{n+1}, z))^{\frac{1}{1-\lambda}} \rightarrow_* 1~(n\rightarrow \infty ).
\end{equation*}
Hence $d(fz,z)=1$. This implies $fz=z$. Finally, it is easy to show that the fixed point of $f$ is unique.
\end{proof}
\begin{thm}
Let $(X,d)$ be a complete multiplicative metric space. Suppose the mapping $f:X\rightarrow X$ satisfies the contraction condition
\begin{equation*}
d(fx, fy)\leq (d(fx, y)\cdot d(fy, x))^{\lambda}~ \text{for all}~x,y\in X,
\end{equation*}
where $\lambda \in \left[0, \frac{1}{2} \right)$ is a constant. Then $f$ has a unique fixed point in $X$. And for any $x\in X$, iterative sequence $(f^nx)$ converges to the fixed point.
\end{thm}
\begin{proof}
Choose $x_0\in X$. Set $x_1=fx_0,x_2=fx_1=f^2x_0,...,x_{n+1}=fx_n=f^{n+1}x_0,...$

\noindent we have
\begin{eqnarray}
d(x_{n+1}, x_n)=d(fx_n,fx_{n-1})&\leq& (d(fx_n,x_{n-1})\cdot d(fx_{n-1},x_n))^{\lambda} \nonumber \\
&\leq&(d(x_{n+1},x_n)\cdot d(x_n,x_{n-1}))^{\lambda}. \nonumber
\end{eqnarray}
Thus we have
\begin{equation*}
d(x_{n+1},x_n)\leq (d(x_n,x_{n-1}))^{\frac{\lambda}{1-\lambda}}= d(x_n,x_{n-1})^h, \nonumber
\end{equation*}
where $h=\frac{\lambda}{1-\lambda}$. For $n>m$,
\begin{eqnarray}
d(x_n, x_m)&\leq& d(x_n,x_{n-1})\cdot d(x_{n-1}, x_{n-2})\cdot \cdot \cdot d(x_{m+1}, x_m) \nonumber \\
&\leq&d(x_1,x_0)^{h^{n-1}+h^{n-2}+\cdot \cdot \cdot + h^m} \leq d(x_1,x_0)^{\frac{h^m}{1-h}} \nonumber 
\end{eqnarray}
This implies $d(x_n,x_m)\rightarrow_* 1(n,m\rightarrow \infty)$. Hence $(x_n)$ is a Cauchy sequence. By the multiplicative completeness of $X$, 
there is $z\in X$ such that $x_n\rightarrow_* z~(n\rightarrow \infty)$. Since 
\begin{eqnarray}
d(fz, z)&\leq& d(fx_n, fz)\cdot d(fx_n, z) \nonumber \\
&\leq& (d(fz, x_n)\cdot d(fx_{n}, z))^{\lambda} \cdot d(x_{n+1}, z), \nonumber \\
&\leq&(d(fz, z) \cdot d(x_n, z)\cdot d(x_{n+1}, z))^{\lambda} \cdot d(x_{n+1}, z) \nonumber 
\end{eqnarray}
we have
\begin{equation*}
d(fz, z)\leq ((d(x_{n+1},z) \cdot d(x_n, z))^{\lambda}\cdot d(x_{n+1}, z))^{\frac{1}{1-\lambda}} \rightarrow_* 1~(n\rightarrow \infty ).
\end{equation*}
Hence $d(fz,z)=1$. This implies $fz=z$. So $z$ is a fixed point of $f$. Finally, it is easy to show that the fixed point of $f$ is unique.
\end{proof}
\begin{rem}
Theorems 3.2, 3.5 and 3.6 carry some fixed point theorems of contraction mappings in metric spaces to multiplicative metric spaces.
\end{rem}

We conclude with  the following two examples: \newline
 Let $X=\left\{   (x,1)\in \mathbb{R}^2:1\leq x \leq 2 \right\} \cup \left\{   (1,x)\in \mathbb{R}^2:1\leq x \leq 2 \right\}$. Consider a mapping
$d:X\times X \rightarrow \mathbb{R}$ defined by 
\begin{equation*}
d((a,b),(c,d))=\left(\left| \frac{a}{c}\right|^*\cdot \left| \frac{b}{d}\right|^*\right)^{\frac{1}{3}}
\end{equation*}
Then $(X,d)$ is a complete multiplicative metric space. Let $f:X\rightarrow X$ be a mapping given as follows
\begin{equation*}
f(x,1)=(1,\sqrt{x})~\text{and}~f(1,x)=(\sqrt{x}, 1)
\end{equation*} 
Thus the mapping holds the  following multiplicative contraction condition 
\begin{equation*}
d(f(a,b),f(c,d))\leq d((a,b),(c,d))^{\lambda}~\text{for all}~(a,b),(c,d)\in X,
\end{equation*}
with constant $\lambda=\frac{1}{2}\in [0,1)$. It is obvious that $f$ has a unique fixed point $(1,1)\in X$.

If we let $X=[0.1, 1]$, then $X$ is a complete multiplicative metric space with respect to the multiplicative metric $| . |^{*}$. Thus a mapping $f:X\rightarrow X$ defined as $f(x)=e^{x-1-\frac{x^3}{10}}$ satisfies the following multiplicative contraction:
\begin{equation}
\left| \frac{f(x)}{f(y)}\right |^{*}\leq \left( \left| \frac{x}{y} \right |^{*}\right)^{\lambda} \text{for all}~x, y\in X,
\end{equation}
where $\lambda=0.997$. Finally,  we can see that $f$ has a unique fixed point $ 0.7411317711\in X $.

\end{document}